\documentclass[11pt]{amsart}
\usepackage{amsmath}
\usepackage{amsfonts}
\usepackage{amsthm}
\usepackage{amssymb}
\usepackage{mathrsfs}
\usepackage{latexsym}
\usepackage{verbatim}
\usepackage{a4wide}
\usepackage{enumerate}
\usepackage{lipsum}
\usepackage[all]{xy}
\usepackage{hyperref}

\numberwithin{equation}{section}

\renewcommand{\thetheoremName}

\newcommand{\IC}{\mathbb{C}}
\newcommand{\IF}{\mathbb{F}}

\newcommand{\IN}{\mathbb{N}}
\newcommand{\IG}{\mathbb{G}}

\newcommand{\IQ}{\mathbb{Q}}

\newcommand{\IZ}{\mathbb{Z}}

\newcommand{\calF}{\mathcal{F}}

\newcommand{\calO}{\mathcal{O}}

\newcommand{\calS}{\mathcal{S}}
\newcommand{\calT}{\mathcal{T}}

\newcommand{\ip}{\mathfrak{p}}

\def\Hom{\mathrm{Hom}}
\def\End{\mathrm{End}}
\def\Aut{\mathrm{Aut}}

\def\GL{\mathrm{GL}}

\def\Gal{\mathrm{Gal}}

\def\Spec{\mathrm{Spec}}
\def\Spf{\mathrm{Spf}}

\def\Jac{\mathrm{Jac}}

\newtheorem{theorem}{Theorem}[section]
\newtheorem*{theorem*}{Theorem}

\newtheorem{question}[theorem]{Question}

\newtheorem{lemma}[theorem]{Lemma}
\newtheorem{proposition}[theorem]{Proposition}
\newtheorem{corollary}[theorem]{Corollary}
\newtheorem*{conj}{Conjecture}

\theoremstyle{definition}
\newtheorem{definition}[theorem]{Definition}

\theoremstyle{remark}

\newtheorem{remark}[theorem]{Remark}

\begin{document}

\title{An infinitesimal $p$-adic multiplicative Manin-Mumford Conjecture}
\author{Vlad Serban}

\address{Vlad Serban, Department of Mathematics, University of Vienna, 1 Oskar-Morgenstern-Platz, 1090 Vienna, Austria}
\email{vlad.serban@univie.ac.at}


\begin{abstract}
Our results concern certain analytic functions on the open unit poly-disc in $\mathbb{C}^n_p$ centered at the multiplicative unit and we show such functions only vanish at finitely many $n$-tuples of roots of unity $(\zeta_1-1,\ldots,\zeta_n-1)$ unless they vanish along a translate of the formal multiplicative group. For polynomial functions, this follows from the multiplicative Manin-Mumford conjecture. However we allow for a much wider class of analytic functions; in particular we establish a rigidity result for formal tori. Moreover, our methods apply to Lubin-Tate formal groups beyond just formal $\mathbb{G}_m$ and we extend the results to this setting.
\end{abstract}

\thanks{The author would like to extend special gratitude to Frank Calegari for suggesting this project and for his support throughout. The author would also like to thank Patrick Allen, Haruzo Hida, Richard Moy, Joel Specter, Jacob Tsimerman, Paul VanKoughnett, Felipe Voloch as well as the referee for helpful conversations and comments on previous drafts of this paper.\\
This research was partially supported by the Fields Institute for Research in Mathematical Sciences. Its contents are solely the responsibility of the author. Link: \url{www.fields.utoronto.ca}.
} 
\bibliographystyle{alpha}

\maketitle

\section{Introduction}

The classical Manin-Mumford conjecture, which was proven by M. Raynaud \cite{MR688265}, states that for an algebraic curve $C$ of genus greater than one defined over a number field $K$ together with an embedding defined over $K$ of $C$ into its Jacobian, there are only finitely many torsion points of the Jacobian on the curve, i.e. $C(\overline{K})\cap \Jac(C)(\overline{K})_{\text{tor}}$ is a finite set.
One may ask a similar question replacing the Jacobian by an abelian variety, or more generally a commutative group variety $G$. The case of $G=\mathbb{G}_m^n$, the so-called multiplicative Manin-Mumford conjecture, was already considered by S. Lang in \cite{MR0130219} and \cite{MR0190146}, where he mentions proofs due to Y. Ihara, J-P. Serre and J. Tate. The result states that 
if an irreducible curve $C$ embedded in $\mathbb{G}_m^n$ contains infinitely many torsion points, it must be a translate of $\IG_m$ by a torsion point. Considering the case of $n=2$ for ease of exposition, this amounts to the following explicit statement on polynomials:
\begin{proposition}[\cite{MR0190146}, p.230]
Let $C$ be an absolutely irreducible plane curve given by the zero set of a polynomial $f(X,Y)=0$. If $C$ passes through the multiplicative origin and for infinitely many roots of unity $(\zeta, \xi)$
$$f(\zeta,\xi)=0,$$
then up to constant $f(X,Y)=X^m-Y^{l}$ or $f(X,Y)=X^mY^{l}-1$ for a pair of nonnegative integers $(m,l)\neq (0,0)$. 
\end{proposition}
 The proof relies on the algebraic properties of the polynomial $f$. However, in the $p$-adic setting, we obtain the following statement for power series as a special case of our results:
 \begin{proposition}\label{twovarintro}
 Let $\mathcal{O}_F$ denote the ring of integers of a finite extension $F/\mathbb{Q}_p$ and $\phi\in\mathcal{O}_F[[X,Y]]$ an irreducible power series passing through the origin. If 
 $$\phi(\zeta-1,\xi-1)=0$$
  for infinitely many pairs of $p$-power roots of unity $(\zeta, \xi)$, then after possibly switching $X$ and $Y$ there is $m\in \mathbb{Z}_p$ so that $\phi=(X+1)^m-(Y+1)$, up to multiplying by a unit, where $(X+1)^m=1+\sum_{i=1}^{\infty}\frac{m\cdots(m-i+1)}{i!}X^i$.
 \end{proposition}
 
 Denoting by $\calS$ the set $\{\zeta-1\vert \zeta\in\mu_{p^\infty}(\overline{\IQ}_p)\}$ of torsion points of the formal multiplicative group, we consider more generally the set of $n$-tuples $\calS^n$. We prove the following and deduce Proposition \ref{twovarintro} when $n=2$ and $I=(\phi)$:
\begin{theorem}\label{thm:intro}
  Let $A=\calO_F[[X_1,\ldots, X_n]]/I$, where $\mathcal{O}_F$ denotes the ring of integers of a finite extension of $\mathbb{Q}_p$. 
\begin{enumerate}
\item If the formal scheme $\Spf(A)$ contains infinitely many torsion points in $\calS^n$ then it contains a translate by a torsion point of a formal torus $\widehat{\IG}_m^k$ for some $k>0$.
\item There exists in $\Spf(A)$ a finite union $\mathcal{T}$ of translates by torsion points of formal subtori and an explicit constant $C_I>0$ depending on $I$ and the choice of $p$-adic absolute value $\vert-\vert_p$ such that for any $\phi\in I$,
$$\vert\phi(\zeta_1-1,\ldots, \zeta_n-1)\vert_p> C_I $$
provided $(\zeta_1-1,\ldots, \zeta_n-1)\in \calS^n\setminus (\mathcal{T}(\overline{\IQ}_p)\cap \calS^n)$.
\end{enumerate}
\end{theorem}
Power series $\phi \in \calO_F[[X_1,\ldots, X_n]]$ give rise to analytic functions on the open $p$-adic polydisk $\mathbb{B}^n(e,1)$ in $\IC_p^n$ of radius one centered at $e=(1,\ldots, 1)$ by evaluating points $Q\in\mathbb{B}^n(e,1)$ at $Q-e$. The only roots of unity on the open disk have $p$-power order. The proof of Theorem \ref{thm:intro}, given in Section \ref{sec:gmcase}, relies on the fact that any infinite sequence in $\calS^n$ must approach the boundary of the polydisk since the normalized valuation is $v_p(\zeta_{p^k}-1)=1/(p^k-p^{k-1})$ when $\zeta_{p^k}$ has exact order $p^k$. We also utilize the action of automorphisms of the formal group $\widehat{\IG}^n_m$ on the torsion points in $\calS^n$ together with the algebraic properties of formal power series rings. \par
In particular, for $m\in\IN$ this gives a new proof of Manin-Mumford for tori and $m$-primary torsion that does not rely on a large Galois orbit argument for special points. Moreover, our methods can be used to prove that Theorem \ref{thm:intro} holds in greater generality replacing $\widehat{\IG}_m$ by a Lubin-Tate formal group $\calF_{LT}$ and replacing $\calS^n$ by $n$-tuples of torsion points of $\calF_{LT}$. This is shown in Section \ref{LT}, whereas similar results for formal groups of abelian schemes will be addressed in future work. 
\par
Many generalizations of the Manin-Mumford conjecture are known by work of S. Zhang \cite{MR1254133,MR1609518}, E. Ullmo \cite{MR1609514} and others, however to our knowledge these infinitesimal $p$-adic strengthenings have not previously been considered. Rather, they are related to rigidity results appearing in H. Hida's work which we discuss in Section \ref{rigidity}. We remark that Theorem \ref{thm:intro} makes the additional claim that there cannot be infinitely many torsion points arbitrarily close to $\Spf(A)(\IC_p)\subset \mathbb{B}^n(0,1)$ unless for a specific geometric reason. This is a purely $p$-adic phenomenon, for instance torsion points on abelian varieties are dense even in the complex analytic topology. It was observed by J. Tate and F. Voloch \cite{MR1405976} that given a linear form $f$ vanishing at $Z(f)$ and a choice of $p$-adic absolute value $\vert-\vert_p$, there is a uniform bound $\varepsilon_f$ such that for any $n$-tuple of roots of unity,
$$f(\zeta_1,\ldots, \zeta_n)\neq 0\Rightarrow \vert (\zeta_1,\ldots, \zeta_n)-P\vert_p >\varepsilon_f\text{ } \forall P\in Z(f).$$
They formulated a general conjecture for the $p$-adic distance from torsion points on a semi-abelian variety over $\IC_p$ to a subvariety which was proven by T. Scanlon \cite{MR1631061} when the semi-abelian variety is defined over a finite extension of $\IQ_p$. Our results exhibit the same phenomenon for formal power series and $p$-power roots of unity, whereas in her thesis A. Neira \cite{MR2705372} proves this for analytic functions on the closed disk. In this case one need not restrict to $p$-power roots of unity. Finally, P. Monsky \cite{MR614398} studied such $p$-adic power series rings with applications to Iwasawa theory in mind and \cite[Section 2]{MR614398} established some of the results used in this paper.
\par
 The rings of analytic functions we consider are related to many interesting arithmetic objects. They occur naturally as completed group rings such as $\mathbb{Z}_p[[\mathbb{Z}_p^n]]\cong \mathbb{Z}_p[[X_1,\ldots, X_n]]$, Iwasawa theory studies the ideals cut out by $p$-adic $L$-functions inside these rings, and they are completed local rings at smooth points of schemes over $\mathcal{O}_F$. Our initial interest in the problem was motivated by the study of families of $p$-adic automorphic forms parametrized by weight. The spaces of $p$-adic weights are up to connected components formal power series rings and one is led to consider a slightly larger class of special points corresponding to weights of classical automorphic forms. The applications to this setting constitute the content of a subsequent paper.
\par
Finally, we remark that all the results in this paper hold more generally when $\calO_F$ is replaced by a subring of $\calO_{\IC_p}$ that is discretely valued. That condition is also necessary, as even a one-variable result along the lines of Lemma \ref{smallpoints} fails for $\calO_{\IC_p}$-coefficients. The restriction to integral coefficients is also necessary, and the $p$-adic logarithm gives an example of an analytic function on the open unit ball vanishing at all of $\calS$.

\section{Manin-Mumford for formal $\IG_m$}\label{sec:gmcase}
Throughout this paper, we fix a prime $p$. We denote by $\calO_F$ the ring of integers of a finite extension $F$ of the $p$-adic numbers $\IQ_p$ and denote by $v_p$ the valuation on the $p$-adic complex numbers $\IC_p$, normalized so that $v_p(p)=1$. \par

\subsection{Formal schemes and rigidity results}\label{rigidity}
We review some relevant rigidity theorems found in the literature. Over a field $k$ of characteristic $p$, C-L. Chai \cite{MR2439259} proves a rigidity result for $p$-divisible formal groups, which is used in his work together with F. Oort on the Hecke Orbit conjecture for Siegel modular varieties (see e.g. \cite{MR2159378}). Considering the torus 
$$ \widehat{\IG}_{m/k}^n=\Spf(k[[X_1,\ldots, X_n]]),$$
let $X_{*}(\widehat{\IG}_{m}^n)=\Hom_k(\widehat{\IG}_{m}, \widehat{\IG}_{m}^n)\cong \IZ_p^n$ denote the group of cocharacters, so that $\GL(X_*{})\cong \GL_n(\IZ_p^n)$ naturally acts on the torus. 
One has from \cite[Theorem 4.3]{MR2439259}:
\begin{theorem}[Chai]
Let $k=\overline{\IF}_p$ and $Z\subset \widehat{\IG}_{m/k}^n$ a closed formal subscheme, equidimensional of dimension $r$.
If $Z$ is stable under the diagonal action for all $u\in U$ in an open subgroup of $(\IZ_p^\times)^n\subset \GL(X_{*})$, then
there are finitely many $\IZ_p$-direct summands $T_1,\ldots,T_s$ of rank $r$ of $X_{*}(\widehat{\IG}_{m}^n)$ so that
$$Z=\bigcup_{i=1}^s\widehat{\IG}_{m/k}\otimes T_i.$$
\end{theorem}

Hida uses Chai's rigidity results \cite[Section 3.4]{MR2680417} and establishes characteristic zero versions thereof in \cite[Lemma 1.2]{MR2782844} and \cite[Section 4]{MR3220926}. He proves:
\begin{lemma}[\cite{MR3220926}, Lemma 4.1]\label{Hidaslemma}
Let $Z=\Spf(\calO_F[[X_1,\ldots, X_n]]/I)$ be a closed formal subscheme of $ \widehat{\IG}_m^n$ that is flat and geometrically irreducible. Suppose there is an open subgroup $U\subseteq \IZ_p^\times$ such that $Z$ is stable under the action $(1+X_i)\mapsto (1+X_i)^u$ for all $u\in U$. If there exists a subset $\Omega\subseteq Z(\IC_p)\cap \mu_{p^\infty}^n(\IC_p)$ Zariski dense in $Z$, then $Z$ is the translate of a formal subtorus by a torsion point in $\Omega$. 
\end{lemma}

In particular, he obtains a rigidity result for formal power series by applying Lemma \ref{Hidaslemma} to their graph in \cite[Corollary 4.2]{MR3220926}:

\begin{corollary}[Hida]\label{Hidacor}
Let $\phi\in \calO_F[[X_1,\ldots, X_n]]$ be a power series such that there is a Zariski dense subset $\Omega\subset \mu_{p^\infty}^n(\IC_p)$ in $\widehat{\IG}_m^n(\IC_p)$ with $\phi(\zeta-1)\subseteq \mu_{p^\infty}(\IC_p)\text{ for all }\zeta \in \Omega$.
Then there exist $\zeta_0\in \mu_{p^\infty}(\calO_F) $ and $N=(N_1, \ldots, N_n) \in \IZ_p^n$ such that $\phi(X_1,\ldots, X_n)=\zeta_0\prod_{i=1}^n (1+X_i)^{N_i}$. 
\end{corollary}

Writing $A=\calO_F[[X_1,\ldots, X_n]]/I$ and assuming $\Spf(A)$ is geometrically irreducible, we therefore consider the following statements:
\begin{enumerate}[(I)]
\item \label{one} The formal subscheme $\Spf(A)\subset \widehat{\IG}_m^n$ is the translate of a formal subtorus by a torsion point.
\item \label{two} There is a Zariski-dense set of torsion points of $\widehat{\IG}_m^n$ on $\Spf(A)(\overline{\IQ}_p)$.
\item \label{three} The formal subscheme $\Spf(A)\subset \widehat{\IG}_m^n$ is stable under the action of an open subgroup $U$ of the diagonal in the cocharacters $\GL(X_{*})\cong \GL_n(\IZ_p)$.
\end{enumerate}
The first statement implies the two others. Chai's result establishes a version of (\ref{three}) $\Rightarrow$ (\ref{one}) in characteristic $p$. In characteristic zero, Lemma \ref{Hidaslemma} shows that (\ref{two} \& \ref{three}) $\Rightarrow$ (\ref{one}), whereas Corollary \ref{Hidacor} shows in some special cases that (\ref{two}) $\Rightarrow$ (\ref{one}). In Proposition \ref{mainresult} we prove a vanishing result for arbitrary formal power series and deduce in all generality that (\ref{two}) $\Rightarrow$ (\ref{one}), see in particular Corollary \ref{cor:formal}. 
To this end, we adopt the point of view of ``unlikely intersection'' results such as the Manin-Mumford and Andr\'e-Oort conjectures.
\subsection{Manin-Mumford formulation}
We use terminology inspired by the following formulation in \cite[Section 3.1.]{MR2290497} of the classical Manin-Mumford conjecture: 
Let $X/\IC$ be an algebraic variety. 
Define a set of \emph{special subvarieties} $S_X$ to be the following irreducible subvarieties of $X$:
\begin{itemize}
\item If $X$ is an abelian variety, the special subvarieties are the translates by torsion points of abelian subvarieties of $X$.
\item If $X$ is a torus, the special subvarieties are given by the products of torsion points with subtori.
\end{itemize}
A \emph{special point} is a zero-dimensional special subvariety.
The conjecture may then simply be stated as:
\begin{conj}[Manin-Mumford]\label{MMM}
 An irreducible component of the Zariski closure of a set of special points is a special subvariety.
\end{conj}

We are interested in the $\overline{\IQ}_p$-points on $\Spec(\calO_F[[X_1,\ldots,X_n]])$ with coordinates:
$$\calS^n=\{(\zeta_1-1,\ldots,\zeta_n-1)\vert \zeta_i\in\mu_{p^\infty}(\overline{\IQ}_p)\},$$
which are precisely the torsion points of the formal Lie group $\widehat{\IG}_m^n$.
For $A=\calO_F[[X_1,\ldots, X_n]]/I$ we define the set of \emph{special points} $\calS_A\subset\Spec(A)$ to be the points of $\calS^n$ lying on $\Spec(A)$, namely prime ideals of the form $\ip=(X_1+1-\zeta_1,\ldots, X_n+1-\zeta_n)$ containing $I$.
We want $\Spec(A)$ to be an irreducible \emph{special subscheme} exactly when $\Spf(A)\subset\widehat{\IG}_m^n$ is the product of a formal subtorus by a torsion point of $\widehat{\IG}_m^n$. Note that endomorphisms of the formal group law act on a choice of coordinates via 
$$(X_1, \ldots, X_n)\mapsto\left(\prod_{j=1}^n(X_j+1)^{a_{1,j}}-1, \ldots, \prod_{j=1}^n(X_j+1)^{a_{n,j}}-1\right)$$
for matrices $(a_{i,j})\in M_n(\IZ_p)\cong  \End(\widehat{\IG}_m^n)$. 
\par
For any choice of coordinates, the set of special subschemes should account for twists by automorphisms in $\GL_n(\IZ_p)$. We therefore make the following definitions:

\begin{definition} 
\leavevmode
\begin{enumerate}
\item A \emph{multiplicative change of variables} on $\calO_F[[X_1,\ldots, X_n]]$ is given by possibly swapping the roles of $X_i$ and $X_j$ and a series of transformations of the form:
$$X_i\mapsto (1+X_i)\prod_{1\leq j< i} (1+X_j)^{B_{ij}}-1$$
 for $1\leq i\leq n $, where $B_{ij}\in \mathbb{Z}_p$.
\item A \emph{special (multiplicative) subscheme} of $\Spec (\calO_F[[X_1,\ldots, X_n]])$ is a closed affine subscheme that after a multiplicative change of variables becomes a finite union of intersections of hyperplanes of the form $X_i=\zeta_i-1$ for roots of unity $\zeta_i\in\mu_{p^\infty}$.
\end{enumerate}
\end{definition}

We can now state the $p$-adic infinitesimal result as follows:

\begin{theorem}
Let $A=\calO_F[[X_1,\ldots, X_n]]/I$. An irreducible component of the Zariski closure of the special points $\calS_A$ on $\Spec(A)$ is a special multiplicative subscheme.
\end{theorem}

We also establish the stronger result in this $p$-adic setting:

\begin{theorem}\label{strongversion}
 Let $A=\calO_F[[X_1,\ldots, X_n]]/I$. There exists a finite union of special subschemes contained in $\Spec(A)$, denoted $\mathcal{T}$, and a constant $C_I>0$ depending only on $I$ and the choice of $p$-adic absolute value $\vert-\vert_p$, such that  
$$\vert\phi(\zeta_1-1,\ldots, \zeta_n-1)\vert_p> C_I $$
for any $\phi\in I$, provided $(\zeta_1-1,\ldots, \zeta_n-1)\in \calS^n\setminus (\mathcal{T}(\overline{\IQ}_p)\cap \calS^n)$.
\end{theorem}

\subsection{Almost vanishing loci}

Let $\pi$ denote a uniformizer for $\calO_F$ and $\IF_q=\calO_F/\pi\calO_F$ the residue field. For any ideal $I\subset \calO_F[[X_1,\ldots, X_n]]$ and $\varepsilon>0$, we consider special points $\varepsilon$-close to $\Spec(\calO_F[[X_1,\ldots,X_n]]/I)$:
$$S_I(\varepsilon):=\{(\zeta_1-1,\ldots,\zeta_n-1)\in \calS^n\vert\forall \phi\in I \text{, }\vert \phi(\zeta_1-1,\ldots,\zeta_n-1)\vert_p<\varepsilon \}. $$
We note that $\calO_F[[X_1,\ldots, X_n]]$ is a Noetherian ring and that in this ultrametric setting it suffices to check the conditions on the finitely many generators of $I$. If $I=(\phi)$, we simply write $S_\phi(\varepsilon)$. 

The endomorphisms of $\widehat{\IG}_m^n$ transform $S_I(\varepsilon)$. In particular, performing a multiplicative change of variables on $\calO_F[[X_1,\ldots,X_n]]/I$ given by a lower triangular $(B_{ij})\in\GL_n(\IZ_p)$ acts on $S_I(\varepsilon)$ via the inverse automorphism of $\widehat{\IG}_m^n$, which we denote by $(B_{ij}^*)\in\GL_n(\IZ_p)$, as follows:
\begin{equation}\label{equation:action}\begin{pmatrix}
   1 & 0 & \dots  &0 \\
    B_{21}^* & 1 & \dots  & 0 \\
        \vdots & \vdots &  \ddots & \vdots \\
    B_{n1}^* &B_{n2}^* & \dots  & 1
    \end{pmatrix}
    :
    \begin{pmatrix}
   \zeta_1-1\\
   \zeta_2-1\\
    \vdots \\
    \zeta_n-1
    \end{pmatrix}
    \mapsto
    \begin{pmatrix}
   \zeta_1-1\\
   \zeta_2\zeta_1^{B_{21}^*}-1\\
    \vdots \\
    \zeta_n\prod_{1\leq i <n}\zeta_i^{B_{ni}^*}-1
    \end{pmatrix}
   \end{equation}
To understand when $S_I(\varepsilon)$ can be infinite, we may after twisting by an automorphism arrange for an explicit parametrization. Swapping coordinates if necessary so that orders of coordinates decrease lexicographically, any element of $\calS^n$ may be written as $(\zeta-1,\zeta^{a_{2}}-1,\ldots,\zeta^{a_{2}a_{3}\dots a_{n}}-1)$ for some $\zeta\in  \mu_{p^\infty}$ and exponents $a_i\in\IZ_p$. We show:

\begin{lemma}\label{nicesequence}
For any $\varepsilon,\delta >0$ and any infinite sequence in $S_I(\varepsilon)$ given by
$$\{(\zeta_k-1,\zeta_k^{a_{2k}}-1, \zeta_k^{a_{2k}a_{3k}}-1,\ldots,\zeta_k^{a_{2k}a_{3k}\dots a_{nk}}-1)\vert \zeta_k\in \mu_{p^\infty}\}_{k\in\IN}$$
with exponents $\{a_{ik}\}_{k\in\IN}\in \mathbb{B}_{\IZ_p}(A_i, \delta)$ in the $p$-adic ball centered at some $A_i\in\IZ_p$,
we may after a multiplicative change of variables arrange for $a_{ik}\in \mathbb{B}_{\IZ_p}(0, \delta)$. 
\end{lemma}
\begin{proof}
There exist $B_{ij}\in \IZ_p$ such that the inverse matrix $(B_{ij}^*)$ satisfies
$$\begin{pmatrix}
   1 & 0 & \dots  &0 \\
    B_{21}^* & 1 & \dots  & 0 \\
        \vdots & \vdots &  \ddots & \vdots \\
    B_{n1}^* &B_{n2}^* & \dots  & 1
    \end{pmatrix}
    \cdot
    \begin{pmatrix}
   1\\
   A_2\\
    \vdots \\
    A_2\cdots A_{n}
    \end{pmatrix}
    =
    \begin{pmatrix}
  1\\
   0\\
    \vdots \\
    0    \end{pmatrix}$$
    and we twist by the multiplicative change of variables given by the $B_{ij}$. It follows from \eqref{equation:action} that all the exponent sequences $a_{2k}\dots a_{jk}$ land in $\mathbb{B}_{\IZ_p}(0, \delta)$ and in particular $A_2=0$ and $a_2\in \mathbb{B}_{\IZ_p}(0, \delta)$. Now apply the same procedure to the $(n-1)$ last coordinates $(\xi_k-1, \xi_k^{a_{3k}}-1,\ldots,\xi_k^{a_{3k}\dots a_{nk}}-1)$, where $\xi_k=\zeta_k^{a_{2k}}$, to arrange for $A_3=0$. Iterating this process, we get $A_i=0$ for $2\leq i\leq n$ as claimed.
\end{proof}

We turn to the proof of the general rigidity result for power series:
\begin{proposition}\label{mainresult}
Let $\phi\in \calO_F[[X_1,\ldots, X_n]]$ be a power series. There exists $\varepsilon>0$ and a finite set of irreducible special multiplicative subschemes $S_1,\ldots S_K$ such that $\phi$ vanishes along $S_k$ for $1\leq k\leq K$ and such that
$$S_\phi(\varepsilon)\setminus\cup_{k=1}^KS_k=\emptyset.$$

\end{proposition}

The main theorem of this section follows:
\begin{proof}[Proof of Theorem \ref{strongversion}]
The ideal $I\subset \calO_F[[X_1,\ldots, X_n]]$ is finitely generated. Writing $I=(\phi_1,\ldots, \phi_d)$, we conclude since $S_I(\varepsilon)= \cap_{k=1}^{d}S_{\phi_k}(\varepsilon)$ and the set of special subschemes is closed under intersection.
\end{proof}
To establish Proposition \ref{mainresult}, the following lemma deals with one-variable power series.

\begin{lemma}\label{smallpoints}
Let $\phi\in \calO_F[[X]]$ be a power series. There exists $\delta'>0$ such that, for any $\delta'\geq\delta>0$, if $x\in\mathfrak{m}_{\IC_p}=\mathbb{B}(0,1)$ satisfies
$$v_p(x)< \delta \text{ and }v_p(\phi(x)) > \delta^{-1}\cdot v_p(x), $$
then $\phi\in \pi\calO_F[[X]]$.
\end{lemma}
\begin{proof}
Assume $\phi  \not\in \pi\calO_F[[X]]$. 
Writing $\phi=\sum_{i=0}^\infty a_iX^i$, let $M$ be the smallest integer such that $v_p(a_M)=0$.
Then, provided $v_p(x)<v_p(\pi)/M$, we have 
$$v_p(\phi(x))=v_p(a_Mx^M)=Mv_p(x).$$
Thus taking $\delta'=v_p(\pi)/M$ yields a contradiction.
\end{proof}

The next lemma is crucial and provides the reduction step in our proof. It puts a strong restriction on the subsets of $\calS^n$ that can be realized as $S_{\phi}(\varepsilon)$ for some $\varepsilon>0$ and $\phi\in\calO_F[[X_1,\ldots, X_n]]$. 

\begin{lemma}\label{factoroutpi}
For any power series $\phi \in\calO_F[[X_1,\ldots, X_n]]$ and constant $0\leq c<1$, there exists $\delta>0$ such that either:
\begin{enumerate}
\item the power series $\phi\in\pi\calO_F[[X_1,\ldots, X_n]]$ or
\item the projection to the last coordinate of the set
$$T_\delta:=\{(\zeta-1,\zeta^{a_{2}}-1,\ldots,\zeta^{a_{2}a_{3}\dots a_{n}}-1)\vert \zeta \in \mu_{p^\infty}, \vert a_i\vert_p<\delta\}\cap S_\phi(q^{-(1-c)})$$
for 
$a_i\in \IZ_p$ is finite.
\end{enumerate}
\end{lemma}

\begin{proof}
If $n=1$ the result follows from Lemma \ref{smallpoints}. 
Let now $n\geq 2$. By induction we prove the following stronger claim: for small enough $\delta>0$ the existence of a sequence $\{(\zeta_k-1,\zeta_k^{a_{2k}}-1,\ldots,\zeta_k^{a_{2k}a_{3k}\dots a_{nk}}-1)\}_{k\in\IN}\in T_\delta$ such that the set of orders of $\zeta_k^{a_{2k}a_{3k}\dots a_{nk}}$ is infinite and  such that
$$\vert\phi(\zeta_k-1,\zeta_k^{a_{2k}}-1,\ldots,\zeta_k^{a_{2k}a_{3k}\dots a_{nk}}-1)\vert_p\leq \vert \zeta_k^{a_{2k}a_{3k}\dots a_{nk}}-1\vert_p^{1/\delta}$$
implies $\phi\in\pi\calO_F[[X_1,\ldots, X_n]]$. Indeed, assume $\phi\not\in\pi\calO_F[[X_1,\ldots, X_n]]$. There is then a largest integer $M$ with $X_n^M$ dividing the reduction $\bar{\phi}\in\IF_q[[X_1,\ldots, X_n]]$ modulo $\pi$
and we choose a lift $\psi\in \calO_F[[X_1,\ldots,X_n]]$ with
$$\phi(X_1,\ldots,X_n)\equiv X_n^M \psi(X_1,\ldots, X_n)\mod \pi.$$
Thus, considering for any $\delta>0$ such a sequence in $T_\delta$, and observing that the valuation $v_p(\zeta_k^{a_{2k}a_{3k}\dots a_{nk}}-1)$ becomes arbitrarily small, after passing to a subsequence we may in addition to that arrange for the property
$$(\zeta_k-1,\zeta_k^{a_{2k}}-1,\ldots,\zeta_k^{a_{2k}a_{3k}\dots a_{nk}}-1)\in S_{\psi}(q^{-(1-c/2)}).$$
Moreover, there exists some power series $\theta\in\calO_F[[X_1,\ldots,X_n]]$ such that we may write
$$\psi(X_1,\ldots,X_n)=\psi(X_1,\ldots, X_{n-1},0)+X_n\cdot \theta(X_1,\ldots,X_n).$$ 
Therefore, after possibly passing to a subsequence, one may establish the inequality
$$\vert \psi(\zeta_k-1,\ldots,\zeta_k^{a_{2k}a_{3k}\dots a_{(n-1)k}}-1,0)\vert_p \leq \vert(\zeta_k^{a_{2k}a_{3k}\dots a_{nk}}-1)\vert_p.$$
To prove the claim, we show $\psi(X_1,\ldots, X_{n-1},0)\in \pi\calO_F[[X_1,\ldots, X_{n-1}]]$
is forced for $\delta>0$ small enough. 
It follows that $\psi \equiv X_n\theta \mod \pi$, and hence
$$\phi(X_1,\ldots,X_n)\equiv X_n^{M+1}\theta(X_1,\ldots,X_n) \mod \pi.$$
This then contradicts the maximality of $M$ and establishes the claim. \par 
When $n=2$, Lemma \ref{smallpoints} applied to $\psi(X_1,0)$ shows  $\psi(X_1,0)\in \pi\calO_F[[X_1]]$ provided $\delta>0$ is small enough. When $n>2$, we obtain from $\vert a_{nk}\vert_p<\delta$ the inequality
$$\vert \zeta_k^{a_{2k}a_{3k}\dots a_{nk}}-1\vert_p\leq \vert \zeta_k^{a_{2k}a_{3k}\dots a_{(n-1)k}}-1\vert_p^{1/\delta}$$ and by induction, for $\delta>0$ small enough, obtain $\psi(X_1,\ldots, X_{n-1},0)\in \pi\calO_F[[X_1,\ldots, X_{n-1}]]$.\par
We now fix $\delta>0$ small enough as in the claim and deduce the result. Suppose the projection to the last coordinate of $T_\delta$ is infinite. Since then the absolute values $\vert \zeta^{a_{2}a_{3}\dots a_{n}}-1\vert_p$ approach one, we may construct a sequence $\{(\zeta_k-1,\zeta_k^{a_{2k}}-1,\ldots,\zeta_k^{a_{2k}a_{3k}\dots a_{nk}}-1)\}_{k\in\IN}\in T_\delta$ with  infinite projection onto the last coordinate and $q^{-(1-c)}\leq\vert \zeta_k^{a_{2k}a_{3k}\dots a_{nk}}-1\vert_p^{1/\delta}$. By the claim it follows that $\phi\in\pi\calO_F[[X_1,\ldots, X_n]]$, as desired.
 
\end{proof}

We now prove the main results of this section.
\begin{proof}[Proof of Proposition \ref{mainresult}]
When $n=1$, assume that for any $\varepsilon$ the set $S_\phi(\varepsilon)$ is infinite. In particular, the valuations of elements in $S_\phi(\varepsilon)$ become arbitrarily small for any $\varepsilon>0$. It follows from Lemma $\ref{smallpoints}$ that $\phi$ must be divisible by any power of the uniformizer $\pi$, hence vanishes identically. All the claims therefore hold.\par
For $n>1$, we may without loss of generality assume $\phi$ has a unit coefficient. We claim that for $\varepsilon$ small enough, we may partition $S_\phi(\varepsilon)$ into finitely many sets $S_j(\varepsilon)$ with the property that some multiplicative change of variables makes the projection onto the last coordinate of $S_j(\varepsilon)$ finite. \par
To prove the claim, we may without loss of generality assume the coordinates of $S_\phi(\varepsilon)$ have decreasing orders, and therefore all elements in $S_\phi(\varepsilon)$ admit a parametrization by $\calS\times\IZ_p^{n-1}$ as $(\zeta-1,\zeta^{a_{2}}-1,\ldots,\zeta^{a_{2}a_{3}\dots a_{n}}-1)$ for $\zeta-1\in \calS$ and $a_i\in\IZ_p$. It follows from Lemmas \ref{nicesequence} and \ref{factoroutpi} that for any $A=(A_2,\ldots, A_n)\in \IZ_p^{n-1}$ and $\varepsilon$ small enough, there exists $\delta_A>0$ such that after a multiplicative change of variables mapping $A$ to $0$ the set 
$$T_{A,\delta_A}:=\{(\zeta-1,\zeta^{a_{2}}-1,\ldots,\zeta^{a_{2}a_{3}\dots a_{n}}-1)\vert a_i\in \mathbb{B}_{\IZ_p}(A_{i}, \delta_A)\}\cap S_\phi(\varepsilon)$$
has finite projection to the last coordinate. By compactness of $\IZ_p^{n-1}$ the cover $\IZ_p^{n-1}=\bigcup_{A\in \IZ_p^{n-1}}\mathbb{B}_{\IZ_p}^{n-1}(A,\delta_A)$ admits a finite subcover. It follows that there is a finite number $f$ of  sets of the form $T_{A_j,\delta_{A_j}}$ for $1\leq j\leq f$ covering
$S_\phi(\varepsilon)$ and we may set $S_j(\varepsilon):=T_{A_j,\delta_{A_j}}$, thus proving the claim. 
\par
Therefore, it suffices to prove the proposition replacing $S_\phi(\varepsilon)$ by $S_j(\varepsilon)$ and after twisting and using the claim we may further reduce to when the last coordinate of $S_j(\varepsilon)$ is a singleton $\xi_n-1\in \calS$ and consider the $n-1$ variable power series $\phi(X_1,\ldots, X_{n-1},\xi_n-1)$. If $\phi(X_1,\ldots, X_{n-1},\xi_n-1)$ vanishes identically, then $\phi$ cuts out a special hyperplane and we are done. 
If $\phi(X_1,\ldots, X_{n-1},\xi_n-1)$ is non-zero, we iterate this procedure until we reduce to the $n=1$ case above. \par
We arrive at the conclusion that, for $\varepsilon$ small enough, there exists a finite cover $S_j(\varepsilon)=\cup_{k=1}^K S_{jk}$, where each $S_{jk}$ is a twist by an automorphism of $\widehat{\IG}_m^n$ of the set $\calS^{r_k-1}\times \{\xi_{r_kk}-1\}\times\cdots \times\{\xi_{nk}-1\}$ for some $1\leq r_k\leq n$ and fixed roots of unity $\xi_{ik}$. Moreover, the function $\phi$ vanishes identically along each $S_{jk}$ and it is easy to see that the Zariski closure of the sets $S_{jk}$ are codimension $n-(r_k-1)$ irreducible special subschemes. 
\end{proof}

In this way one deduces explicit rigidity results for power series, for instance when $n=2$ we arrive at the following formulation:
\begin{proposition}\label{twovariables}
 Let $\phi\in\mathcal{O}_F[[X,Y]]$ be an irreducible power series passing through the origin. If $\phi(\zeta-1,\xi-1)=0$
  for infinitely many pairs of $p$-power roots of unity $(\zeta, \xi)$, then $\phi=(X+1)^m-(Y+1)$
  for some $m\in \mathbb{Z}_p$, up to units and possibly switching the roles of $X$ and $Y$.
 \end{proposition}
\begin{proof}
Assume $\phi$ vanishes at infinitely many points in $\calS^2$. By compactness and Lemma \ref{nicesequence} we may after a multiplicative change of variables assume $\phi$ vanishes at a sequence $(\zeta_k-1, \zeta^{a_k}_k-1)\in \calS^2$ for $a_k\to 0$. By Lemma $\ref{smallpoints}$, and unwinding the change of variables, there are fixed $m\in \IZ_p$ and $\xi\in\mu_{p^\infty}$ such that $\phi(X, (Y+1)(X+1)^m-1)$ vanishes at $(\zeta_k-1,\xi-1)$ for an infinite sequence $\{\zeta_k\}_{k\in\IN}$, after possibly switching the roles of $X$ and $Y$.
We deduce that $\phi(\zeta_k-1, \xi\zeta_k^m-1)=0$ for an infinite sequence $\{\zeta_k\}_{k\in\IN}$.
Over $\calO_{F[\xi]}$ we may then write
$$\phi(X,Y)=\phi(X, \xi(X+1)^m-1)+(\xi(X+1)^m-(Y+1))G(X,Y)$$
for some power series $G(X,Y)$. Since $H(X):=\phi(X, \xi(X+1)^m-1)$ vanishes at infinitely many points $\{\zeta_k-1\}_{k\in\IN}$ it follows that $H=0$ and therefore
\begin{equation}\label{equation:factor}
\phi(X,Y)=(\xi(X+1)^m-(Y+1))G(X,Y)
\end{equation}
for some $G(X,Y)\in \calO_{F[\xi]}[[X,Y]]$. 
Taking conjugates under the group $\Gal(F[\xi]/F)$ of order $g$ in \eqref{equation:factor}, it follows that 
\begin{equation}\label{equation:norms}
\phi(X,Y)^g=\prod_{\sigma\in\Gal(F[\xi]/F)}((X+1)^m-\sigma(\xi)(Y+1))\prod_{\sigma\in\Gal(F[\xi]/F)}G^\sigma(X,Y)
\end{equation}
where $G^\sigma(X,Y)$ is obtained from $G$ by acting on the coefficients. The two factors on the right hand side of \eqref{equation:norms} have coefficients in $\calO_F$. Since $\phi$ is irreducible, it must be an irreducible factor of  $\prod_{\sigma}((X+1)^m-\sigma(\xi)(Y+1))$. But requiring $\phi(0,0)=0$ forces $\xi=g=1$ which shows that  $\phi(X,Y)$ is a unit multiple of $(X+1)^m-(Y+1)$, as desired.

\end{proof}

\subsection{Geometric statements}
We are interested in irreducible components $Z$ of the Zariski closure of the set of special points $\calS_A$ on $A=\calO_F[[X_1,\ldots, X_n]]/I$ for $I$ a non-trivial ideal. As irreducible special subschemes correspond to translates of a formal subtorus by a torsion point of $\widehat{\IG}_m^n$, their dimension should be the dimension of the subtorus. While one has to be careful with dimensions in our setting as the ring of coefficients is one-dimensional, it follows from the definitions of $Z$ that $(\pi)\subset\calO_F$ is never an associated prime and we show:

\begin{theorem} \label{geometricMMM}
Let $Z$ be an irreducible component of the Zariski closure of the special points on $\Spec(\calO_F[[X_1,\ldots,X_n]]/I)$. Then $Z$ is a special multiplicative subscheme. Moreover its codimension is the maximum number $r$ of columns with finite projections in the $\Aut(\widehat{\IG}_m^n)\cong \GL_n(\IZ_p)$-orbit of the special points on $Z$.
\end{theorem}
\begin{proof}
It follows from Theorem \ref{strongversion} that the special points on $Z$ are covered by finitely many irreducible special subschemes $S_1,\ldots S_K$ contained in $Z$. Moreover, since the special points are Zariski dense in $Z$,  we must have
$$Z=\cup_{k=1}^KS_k.$$
By irreducibility, we conclude that $K=1$ and $Z$ is an irreducible special subscheme. The statement on codimensions follows from the definitions, since after twisting by an automorphism $Z$ is a codimension $r$ hyperplane.
 \end{proof}

In particular, we note that:
\begin{corollary}
With notations as above, if the Zariski closure of $\calS_A$ is $d$-dimensional, then $\Spec(A)$ contains a $d$-dimensional special subscheme. 
\end{corollary}
We also record the rigidity result for formal schemes:
\begin{corollary}\label{cor:formal}
If the closed formal subscheme $\Spf(A)\subset \widehat{\IG}_m^n$ contains infinitely many $\overline{\IQ}_p$-torsion points of $\widehat{\IG}_m^n$, it contains a translate of a formal subtorus by a torsion point. Moreover, if $\Spf(A)$ is geometrically irreducible and the set of special points $\calS_A$ are Zariski dense, $\Spf(A)$ is precisely the translate of a formal subtorus by a torsion point.
 \end{corollary}
\begin{proof}
The formal schemes associated to special subschemes are torsion translates of formal subtori, so the result follows from Theorem \ref{geometricMMM}. 
\end{proof}


\section{Going beyond formal $\IG_m$}\label{LT}
At this point, the reader may wonder whether similar results hold for a larger class of formal groups than the multiplicative group. In general, we may consider an $n$-dimensional formal Lie group $\calF$ over a complete Noetherian local ring $R$ with residue field of characteristic $p$ and the corresponding connected $p$-divisible group $\calF[p^\infty]$. The special points are then points in $\calF[p^\infty](\overline{K})$, where $K$ is the field of fractions of $R$. One may ask: 
\begin{question}
Is there a class of special subschemes of the formal group $\calF$ such that the analogue of Theorem \ref{strongversion} holds?
\end{question}
\par

We proceed to show that essentially all the results of Section \ref{sec:gmcase} hold replacing $\widehat{\IG}_m$ by a one-dimensional Lubin-Tate formal group $\calF_{LT}$. The case of replacing $\widehat{\IG}_m^n$ by a more general $n$-dimensional formal group like the formal Lie group of an abelian scheme will be addressed in future work. 

\subsection{Products of one-dimensional Lubin-Tate formal groups}

Let $E/\IQ_p$ be a subfield of $F$ with ring of integers $\calO_E$, uniformizer $\pi_E$ and residue field $\IF_{q'}$. Given a power series $f\in \calO_E[[X]]$ with
$$f(X)\equiv \pi X \mod X^2 \text{ and }f(X)\equiv X^{q'}  \mod \pi_E,$$
Lubin and Tate show \cite[Lemma 1]{MR0172878} that for all $a\in \calO_E$ there is a unique power series $[a](X)\in\calO_E[[X]]$ with $[a](X)\equiv aX \mod X^2$ and $f([a](X))=[a](f(X))$. They construct \cite[Theorem 1]{MR0172878} a commutative one-dimensional formal group law $L(X,Y)\in \calO_E[[X,Y]]$ which satisfies
\begin{enumerate}
\item $L([a](X),[a](Y))=[a](L(X,Y))$
\item $L([a](X),[b](X))=[a+b](X)$
\item $[a]([b](X))=[ab](X)$
\item $[\pi_E](X)=f(X)$ and $[1](X)=X$
\end{enumerate}
for all $a,b \in \calO_E$. Up to isomorphism, the group law is independent of the choice of $f\in \calO_E[[X]]$ with the desired properties. We denote by $\calF_{LT}$ the formal group over $\calO_E$ resulting from this construction. In particular, the properties listed above show there is an injective ring homomorphism
$$\calO_E\hookrightarrow \End(\calF_{LT}),$$ 
and the $\overline{\IQ}_p$-torsion points of $\calF_{LT}$ form a divisible $\calO_E$-module. We will abuse notations and write $\calF_{LT}[\pi_E^\infty]$ for the $\overline{\IQ}_p$-points. It follows from the Newton polygon of $[\pi_E](X)$ and the congruence $[\pi_E](X)\equiv X^{q'} \mod \pi_E$ that a torsion point $\zeta\in\calF_{LT}[\pi_E^\infty]$ of exact order $\pi_E^k$ has normalized valuation 
$v_p(\zeta)=1/(q'^k-q'^{k-1})$.  Adjoining torsion points gives a totally ramified abelian extension $E(\calF_{LT}[\pi_E^\infty])$ of $E$. 
As these properties suggest, we may take as \emph{special points} the points in $\calF_{LT}^n[\pi_E^\infty]$ and define as before:
\begin{definition} 
\leavevmode
\begin{enumerate}
\item
An \emph{$\calF_{LT}$-multiplicative} change of variables on $\calO_F[[X_1,\ldots, X_n]]$ is a series of transformations using the formal group law of $\calF_{LT}$:
$$X_i\mapsto L(\cdots L(L(X_i, [B_{i(i-1)}](X_{i-1})), [B_{i(i-2)}](X_{i-2})),\ldots, [B_{i1}](X_{1}))$$
 where $1\leq i\leq n $ and $ B_{ij}\in \calO_E$ for $1\leq j\leq i$,  composed with possibly swapping variables $X_i\leftrightarrow X_j$.
 \item An \emph{$\calF_{LT}$-special subscheme }of $\Spec(\calO_F[[X_1,\ldots, X_n]])$ is a closed subscheme that after an $\calF_{LT}$-multiplicative change of variables becomes a finite union of intersections of hyperplanes $X_i=\zeta_i$ for $\zeta_i\in \calF_{LT}[\pi_E^\infty]$.
 \end{enumerate}
\end{definition}

\begin{remark}
The \emph{$\calF_{LT}$-multiplicative} changes of variables correspond to automorphisms of the Lubin-Tate formal group law. An irreducible scheme $\Spec(\calO_F[[X_1,\ldots, X_n]]/I)$ is $\calF_{LT}$-special if and only if the formal scheme $\Spf(\calO_F[[X_1,\ldots, X_n]]/I)$ is the translate of a formal subtorus $\calF_{LT}^d$ by a torsion point of $\calF_{LT}^n$. When $\calF_{LT}=\widehat{\IG}_m$, the formal group law is given by 
$$L(X,Y)=(X+1)(Y+1)-1$$
and we recover all the previous definitions as a special case. 
\end{remark}

We revisit the proofs of the key Lemmas \ref{nicesequence} and \ref{factoroutpi} in this more general setting. 
Fix an ideal $I\subset \calO_F[[X_1,\ldots, X_n]]$ and let $A=\calO_F[[X_1,\ldots, X_n]]/I$. For any $\varepsilon>0$ we consider as before the special points almost on $\Spec(A)$:
$$S_I(\varepsilon)=\{(\zeta_1,\ldots,\zeta_n)\in \calF_{LT}^n[\pi_E^\infty]\text{ such that }\forall \phi\in I \text{ } \vert \phi(\zeta_1,\ldots,\zeta_n)\vert_p<\varepsilon\}. $$

As we did for $\widehat{\IG}_m$, we shall use the fact that after possibly swapping coordinates any special point may be written as $(\zeta, [a_{2}](\zeta),\ldots, [a_{2}\cdots a_{n}](\zeta))$ for $a_i\in \calO_E$, together with a compactness argument. This parametrization is still possible since the endomorphism ring of Lubin-Tate formal groups acts transitively on $\pi$-power torsion of lower order for any $\zeta\in \calF_{LT}[\pi_E^\infty]$. Using the properties of the endomorphism ring of $\calF_{LT}^n$ we again obtain:

\begin{lemma}\label{LTnicesequence}
For any $\varepsilon,\delta >0$ and any infinite sequence in $S_I(\varepsilon)$ given by
$$\{(\zeta_k,[a_{2k}](\zeta_k),\ldots, [a_{2k}\cdots a_{nk}](\zeta_k))\vert \zeta_k\in  \calF_{LT}[\pi_E^\infty]\}_{k\in\IN}$$
for $\{a_{ik}\}_{k\in\IN}\in \mathbb{B}_{\calO_E}(A_i, \delta)$ in the $p$-adic ball centered at $A_i\in\calO_E$,
we may after an $\calF_{LT}$-multiplicative change of variables arrange for $a_{ik}\in \mathbb{B}_{\calO_E}(0, \delta)$. 
\end{lemma}

\begin{proof}
As before, choose $B_{ij}\in \calO_E$ such that the coefficients of the inverse change of variables, denoted $B_{ij}^*\in \calO_E$, satisfy the matrix identity
\begin{small}
$$\begin{pmatrix}
   1 & 0 & \dots  &0 \\
    B_{21}^* & 1 & \dots  & 0 \\
        \vdots & \vdots &  \ddots & \vdots \\
    B_{n1}^* &B_{n2}^* & \dots  & 1
    \end{pmatrix}
    \cdot
    \begin{pmatrix}
   1\\
   A_2\\
    \vdots \\
    A_2\cdots A_{n}
    \end{pmatrix}
    =
    \begin{pmatrix}
  1\\
   0\\
    \vdots \\
    0    \end{pmatrix}.$$
    \end{small}
Using crucially that the map $\calO_E\to \End (\calF_{LT})$ is a ring homomorphism, we see that the $\calF_{LT}$-multiplicative change of variables corresponding to the coefficients $B_{ij}$ acts on torsion points via:
    \begin{small}
    \begin{align*}
    \begin{pmatrix}
   \zeta\\
   [a_2](\zeta)\\
    \vdots \\
    [a_2\cdots a_{n}](\zeta)
    \end{pmatrix}
    &\mapsto
    \begin{pmatrix}
   \zeta\\
   L([B_{21}^*](\zeta), [a_2](\zeta))\\
    \vdots \\
    L([B_{n1}^*](\zeta),L(\cdots L([B_{n(n-1)}^*][a_2\cdots a_{n-1}](\zeta), [a_2\cdots a_{n}](\zeta))\cdots)
    \end{pmatrix}
    \\
    &=
   \begin{pmatrix}
   \zeta\\
   [B_{21}^*+a_2](\zeta)\\
    \vdots \\
 [B_{n1}^*+\cdots+B_{n(n-1)}^*a_2\cdots a_{n-1}+a_2\cdots a_{n}](\zeta)
    \end{pmatrix}
    \\
  \end{align*}
  \end{small}
so that after changing variables we may assume the products $a_{2k}\cdots a_{jk}\in \calO_E$ land in $\mathbb{B}(0,\delta)$ for any $2\leq j\leq n$. Now repeat this process as in the proof of Lemma \ref{nicesequence} to conclude.
\end{proof}

\begin{lemma}\label{LTfactoroutpi}
For any power series $\phi \in\calO_F[[X_1,\ldots, X_n]]$ and constant $0\leq c<1$, there exists $\delta>0$ such that either:
\begin{enumerate}
\item the power series $\phi\in\pi\calO_F[[X_1,\ldots, X_n]]$ or
\item the projection to the last coordinate of the set
$$T_\delta:=\{(\zeta,[a_{2}](\zeta),\ldots, [a_{2}\cdots a_{n}](\zeta))\vert \zeta \in \calF_{LT}[\pi_E^\infty], \vert a_i\vert_p<\delta\}\cap S_\phi(q^{-(1-c)})$$ 
for $a_i\in\calO_E$ is finite.

\end{enumerate}
\end{lemma}
\begin{proof}

If $n=1$ the result again follows from Lemma \ref{smallpoints}. 
Let now $n\geq 2$. By induction as in Lemma \ref{factoroutpi} we prove the following claim: for small enough $\delta>0$ the existence of a sequence $\{(\zeta_k,[a_{2k}](\zeta_k),\ldots, [a_{2k}\cdots a_{nk}](\zeta_k))\}_{k\in\IN}$ in $T_\delta$ such that the set of orders of $[a_{2k}\cdots a_{nk}](\zeta_k)$ is infinite and such that
$$\vert\phi(\zeta_k,[a_{2k}](\zeta_k),\ldots, [a_{2k}\cdots a_{nk}](\zeta_k))\vert_p\leq \vert [a_{2k}\cdots a_{nk}](\zeta_k)\vert_p^{\frac{[E:\IQ_p]}{\delta}}$$
implies $\phi\in\pi\calO_F[[X_1,\ldots, X_n]]$. Indeed, assume $\phi\not\in\pi\calO_F[[X_1,\ldots, X_n]]$. As before there is then a largest integer $M$ with $X_n^M$ dividing the reduction $\bar{\phi}\in\IF_q[[X_1,\ldots, X_n]]$ modulo $\pi$
and a lift $\psi\in \calO_F[[X_1,\ldots,X_n]]$ satisfying
$$\phi(X_1,\ldots,X_n)\equiv X_n^M \psi(X_1,\ldots, X_n)\mod \pi.$$
Thus, considering for $\delta>0$ such a sequence in $T_\delta$, for which the valuations $v_p([a_{2k}\cdots a_{nk}](\zeta_k))$ as before become arbitrarily small, after passing to a subsequence we may obtain that $(\zeta_k,[a_{2k}](\zeta_k),\ldots, [a_{2k}\cdots a_{nk}](\zeta_k))\in S_{\psi}(q^{-(1-c/2)})$.
Moreover, there exists some power series $\theta\in\calO_F[[X_1,\ldots,X_n]]$ such that
$$\psi(X_1,\ldots,X_n)=\psi(X_1,\ldots, X_{n-1},0)+X_n\cdot \theta(X_1,\ldots,X_n).$$ 
Therefore, after possibly passing to a subsequence, one may establish the inequality
$$\vert \psi(\zeta_k,[a_{2k}](\zeta_k),\ldots, [a_{2k}\cdots a_{(n-1)k}](\zeta_k),0)\vert_p \leq \vert[a_{2k}\cdots a_{nk}](\zeta_k)\vert_p.$$
To prove the claim, it again suffices to show $\psi(X_1,\ldots, X_{n-1},0)\in \pi\calO_F[[X_1,\ldots, X_{n-1}]]$
for $\delta$ small enough. \par
When $n=2$, Lemma \ref{smallpoints} applied to $\psi(X_1,0)$ shows  $\psi(X_1,0)\in \pi\calO_F[[X_1]]$ provided $\delta>0$ is small enough as before. When $n>2$, we may now consider the inequality $\vert[a_{2k}\cdots a_{nk}](\zeta_k)\vert_p\leq \vert [a_{2k}\cdots a_{(n-1)k}](\zeta_k)\vert_p^{\frac{[E:\IQ_p]}{\delta}}$ and conclude $\psi(X_1,\ldots, X_{n-1},0)\in \pi\calO_F[[X_1,\ldots, X_{n-1}]]$ by induction. This proves the claim. Since the absolute values $\vert [a_{2}\cdots a_{n}](\zeta)\vert_p$ approach one provided the last coordinate of $T_\delta$ is infinite, the proof follows from the claim as in Lemma \ref{factoroutpi}.
\end{proof}

The main results are now deduced in exactly the same way as in Section \ref{sec:gmcase} for $\calF_{LT}=\widehat{\IG}_m$, replacing the parametrization of torsion points by $\IZ_p^{n-1}$ with one by the compact group $\calO_E^{n-1}$. 
Below are the adapted statements of the geometric formulations. The proof is left to the reader.

\begin{theorem}\label{LTgeometricMMM}
Let $A=\calO_F[[X_1,\ldots, X_n]]/I$. A component of the Zariski closure of the $\calF_{LT}$-special points on $\Spec(A)$ is an $\calF_{LT}$-special subscheme.
\end{theorem}

\begin{theorem}\label{LTstrongversion}
 Let $A=\calO_F[[X_1,\ldots, X_n]]/I$. The formal scheme $\Spf (A)$ contains a finite union $\calT$ of translates of formal subtori $\calF_{LT}^k$ by a torsion point of $\calF_{LT}^n$ and there exists a constant $C_I>0$ depending on $I$ and the choice of $p$-adic absolute value $\vert-\vert_p$ such that for any $\phi\in I$, 
$$\vert\phi(\zeta_1,\ldots, \zeta_n)\vert_p> C_I $$
provided $(\zeta_1,\ldots, \zeta_n)\in \calF_{LT}[\pi_E^\infty]^n\setminus (\calT(\overline{\IQ}_p)\cap\calF_{LT}[\pi_E^\infty]^n)$.
\end{theorem}


\nocite{}
\bibliography{MMMbiblio.bib}

\end{document}